\documentclass[12pt]{amsart}
\usepackage{amscd}
\usepackage{amssymb}
\usepackage{a4wide}
\usepackage{amstext}
\usepackage{amsthm}
\usepackage{mathrsfs}
\usepackage{relsize}
\usepackage[final]{hyperref}
\hypersetup{unicode= false, colorlinks=true, linkcolor=blue,
anchorcolor=blue, citecolor=green, filecolor=red, menucolor=blue,
pagecolor=blue, urlcolor=blue} 
\numberwithin{equation}{section}

\newcommand{\RR}{\mathbb{R}}
\newcommand{\Z}{\mathbb{Z}}

\renewcommand{\phi}{\varphi}

\newcommand{\ima}{{\rm Im}\,}

\newcommand{\vep}{\varepsilon}
\newcommand{\co}{\mathbb{C}}
\newcommand{\N}{\mathbb{N}}

\newcommand{\R}{\mathbb{R}}

\newcommand{\ttt}{\mathcal{T}}

\newtheorem{Thm}{Theorem}[section]
\newtheorem{theorem}[Thm]{Theorem}
\newtheorem{lemma}[Thm]{Lemma}
\newtheorem{proposition}[Thm]{Proposition}
\newtheorem{corollary}[Thm]{Corollary}
\newtheorem{remark}[Thm]{Remark}

\begin{document}
\sloppy

\title[]{Zeros of meromorphic functions of the form $\sum\limits_n \dfrac{c_n}{(z-t_n)^2}$}

\author{Anton Baranov}
\address{Department of Mathematics and Mechanics\\ St. Petersburg State University\\
St. Petersburg, Russia}
\email{anton.d.baranov@gmail.com} 
\author{Vladimir Shemyakov}
\address{Department of Mathematics and Mechanics\\ St. Petersburg State University\\
St. Petersburg, Russia}
\email{vladimir.v.shemyakov@gmail.com}







\thanks{The results of Sections 2 and 3 were obtained with the support of Ministry of Science 
and Higher Education of the Russian Federation (agreement No 075-15-2021-602) and of  
Theoretical Physics and Mathematics Advancement Foundation ``BASIS''.
The results of Sections 4 and 5 were 
obtained with the support of the Russian Science Foundation grant 
19-71-30002. }

\begin{abstract} {We study zeros distribution for meromorphic functions of the form
$\sum\limits_n \dfrac{c_n}{(z-t_n)^2}$, where $\sum_n \dfrac{|c_n|}{|t_n|^2} <\infty$. 
We prove an analog of a well-known theorem of Keldysh and discuss a relation
between zero-free functions of this form and second order differential equtions with polynomial coefficients. } 
\end{abstract}

\maketitle


\section{Introduction and main results}
\label{section1}

Representation of meromorphic functions as sums of simple fractions 
and, in particular, the distribution of zeros of such sums are among the basic problems of analysis.  
One of the classical results concerning this problem is due to M.V.~Keldysh (see \cite{keld} or \cite[Chapter V, Theorem 6.2]{gold}).
In what follows we always assume that $\ttt=\{t_n\}_{n=1}^\infty$  is a sequence of distinct complex numbers such that
$|t_n| \to \infty$. Since the problem is translation invariant, we will assume without loss of generality that $0\notin \ttt$
to simplify the formulas. Then the Keldysh theorem can be stated as follows:
if $c_n\in\co$ satisfy
$$
\sum_{n=1}^\infty |c_n| <\infty, \qquad \sum_{n=1}^\infty c_n \ne 0,
$$
and $\ttt$ has finite order (or, at least, finite lower order), then the function
\begin{equation}
f(z) = \sum_n \frac{c_n}{z-t_n} 
\label{mer}
\end{equation}
has infinitely many zeros. Moreover, $\delta(0,f)=0$, where $\delta(0,f)$ denotes the defect of the value 0, 
and the orders of $N(r,1/f)$ and $T(r,f)$ coincide. Here we use the standard notation of the Nevanlinna theory, 
see the formal definitions in Section \ref{prel}. Roughly speaking, this means that $f$ has as many zeros as it has poles. 

Condition $\sum_{n=1}^\infty c_n \ne 0$  in the Keldysh theorem is essential. In the case when $\sum_{n=1}^\infty |c_n| <\infty$ but
$\sum_{n=1}^\infty c_n =0$ the function \eqref{mer} can have infinitely many zeros or no zeros at all.
Note also that the function $f$ is well-defined under a natural milder restriction on the coefficients
\begin{equation}
\sum_n \frac{|c_n|}{|t_n|} <\infty.
\label{conv}
\end{equation}

In 1993 J.~Clunie, A.~Eremenko and J.~Rossi \cite{cer} posed the following problem:
\medskip
\\
{\bf Conjecture} (Clunie, Eremenko, Rossi, 1993). {\it If $c_n>0$ satisfy \eqref{conv}, then the function \eqref{mer}
has infinitely many zeros. }
\medskip

In spite of numerous deep results concerning this problem (see \cite{cer, el, lang, lr}), it is still open in general. 

In \cite{lr} J.K.~Langley and J.~Rossi considered a similar problem for the functions of the form
\begin{equation}
f_m(z) = \sum_{n=1}^\infty \frac{c_n}{(z-t_n)^m}, \qquad \sum_n \frac{|c_n|}{|t_n|^m} <\infty,
\label{mer1}
\end{equation}
where $m\in\N$. It turned out that the cases $m=1$ and $m=2$ are special, while for $m\ge 3$ the problem becomes easier. 
Namely, if $m\ge 3$, then the function $f_m$ has infinitely many zeros, and moreover, $\delta(0,f_m)<1$
\cite[Theorem 1.9]{lr}. 

A simple example shows that this is no longer true when $m=2$. Let 
\begin{equation}
\label{zfr}
f(z) = \frac{a}{\sin^2(bz-c)} = \frac{a}{b^2}\sum_{n\in\mathbb{Z}} \frac{1}{(z-\frac{\pi n+c}{b})^2},
\end{equation}
where $a, b, c\in\co$, $a, b\ne 0$. Then $f$ is of the form \eqref{mer1} with $m=2$ and has no zeros. However, the result
is still true under stronger conditions. 
Langley and Rossi \cite[Theorem 1.10]{lr} showed that if $\ttt$ has a finite exponent of convergence (meaning that $\sum_n |t_n|^{-L}
<\infty$ for some $L>0$) and $c_n$ satisfy \eqref{conv}, then $\delta(0,f)<1$ when
\begin{equation}
f(z) = \sum_{n=1}^\infty \frac{c_n}{(z-t_n)^2}.
\label{mer2}
\end{equation}

We prove two results which complement those of \cite{lr}. Our first observation is that one can prove 
an analog of the Keldysh theorem for the functions of the form \eqref{mer2}. In what follows we always assume that
all coefficients $c_n$ in \eqref{mer2} are nonzero.
 
\begin{theorem}
\label{main1}
Let $\ttt=\{t_n\}$ have finite convergence exponent. 
If  $c_n\in\co$ satisfy $\sum_{n=1}^\infty |c_n| <\infty$ and $\sum_{n=1}^\infty c_n \ne 0$, then for the
function \eqref{mer2} we have $\delta(0,f)=0$, and the orders of $N(r,1/f)$, $N(r,f)$ and $T(r,f)$ coincide, that is, 
$$
\limsup_{r\to \infty} \frac{\log N(r,1/f)}{\log r} = \limsup_{r\to \infty} \frac{\log N(r,f)}{\log r} = 
\limsup_{r\to \infty} \frac{\log T(r,f)}{\log r}.
$$
\end{theorem}

Of course, our assumption on $c_n$ is much stronger than that of \cite[Theorem 1.10]{lr}. However, in our case
we can show that $\delta(0,f)$ is zero. This is not true in general. Let $g$ be an entire function with simple real zeros $t_n$ such that
$$
\frac{1}{g(z)} = \sum_n \frac{1}{g'(t_n)(z-t_n)}, \qquad  \sum_n \frac{1}{|g'(t_n)|}<\infty.
$$
This is a special case of a function of the so-called Krein class (for the definition see Section \ref{difeq}). 
E.g., one can take as $g$ the canonical product with zeros $n^{\alpha}$, $n\in\N$, $\alpha>2$.
Then $f = \big(\frac{1}{g}\big)' =-\frac{g'}{g^2}$ is a function of the form \eqref{mer2} with $\sum_n |c_n| <\infty$,
but $\sum_n c_n =0$. It is easy to see that $N(r, 1/f) \sim N(r,f)/2$ and $\delta(0,f) = 1/2$.
It seems to be an interesting question whether $1/2$ is the sharp upper bound for the defect in
\cite[Theorem 1.10]{lr}. 

Another approach is based on a relation between the functions of the form \eqref{mer2}
with only a finite number of zeros and the solutions of certain second order differential equations 
with polynomial coefficients. 

\begin{proposition}
\label{diff}
Let $f$ be a function of the form \eqref{mer2} with $\sum_n \frac{|c_n|}{|t_n|^2} <\infty$
and let $f$ be of finite order. 
Assume that $f$ has a finite set of zeros and let $P$ be a polynomial whose zeros coincide with zeros of $f$ counting
multiplicities. Then there exist an entire function $g$ such that $f=P/g^2$ and a polynomial $Q$ 
such that $g$ satisfies the differential equation
\begin{equation}
\label{fz}
Pg'' - P'g' +Qg =0.
\end{equation}
In particular, if $f$ has no zeros, then  $f=1/g^2$ and $g$ satisfies the differential equation
\begin{equation}
\label{nz}
g''+Qg=0
\end{equation}
for some polynomial $Q$.
\end{proposition}

Differential equations of the form \eqref{nz} are very much studied. In particular, it is known that:

a) If ${\rm deg}\, Q=m$, then for any nonzero solution $g$ of \eqref{nz} its order $\rho(g)$ equals $\frac{m+2}{2}$.

b) If $Q(z) = \sum_{j=0}^m a_j z^j$ with $m\ge 1$ and $a_m\ne 0$, then all zeros $t_n$ of $g$ approach the critical rays
\begin{equation}
\label{appr1}
L_j = \Big\{\arg \Big(z +\frac{a_{m-1}}{ma_m}\Big)= \frac{2\pi j -\arg a_m}{m+2}\Big\}, \qquad j=0, \dots m+1,
\end{equation}
in a very strong sense. In particular, 
\begin{equation}
\label{appr}
{\rm dist}\, (t_n, \cup_j L_j) \to 0, \qquad n\to \infty. 
\end{equation}
Note that the zeros need not accumulate to each of the rays $L_j$. See \cite[Chapters 4, 5]{laine} and \cite{hei} for details.

Thus, we have several conditions sufficient for existence of zeros of $f$. 

\begin{corollary}
\label{fer}
If a function $f$ of the form \eqref{mer2} is of finite order and either 
the convergence exponent of $\{t_n\}$ is not in $\N\cup \big(\N+\frac{1}{2}\big)$ or 
$\{t_n\}$ do not satisfy \eqref{appr} for any family of rays $L_j$ of the form \eqref{appr1},
then $f$ has at least one zero. 
\end{corollary}

It would be interesting to develop a similar theory for more general equations of the form \eqref{fz}, but we are
not aware of any results of this type.

Our last result shows that in the case when all poles of $f$ are contained in some sufficiently small angle
the function $f$ either has at least one zero or is of the form \eqref{zfr}.

\begin{theorem} 
\label{kr1}
Let $f$ be a function of the form \eqref{mer2} with $\sum_n \frac{|c_n|}{|t_n|^2} <\infty$.
Assume that $f$ is of finite order $\rho(f) = \frac{m+2}{2}$, $m\in \N$, and $f$ has no zeros. 
If for some $\alpha\in (0, \frac{\pi}{m+2})$ 
all points from $\ttt$ except a finite number lie in the union of the angles
$\{|\arg z| <\alpha\} \cup \{|\arg z -\pi| <\alpha\} $, then $m=0$
and $ f(z) = \frac{a}{\sin^2(bz-c)} $ for some $a, b, c\in\co$, $a, b\ne 0$.
\end{theorem}

In particular, if $f$ is of finite order and for any $\vep>0$ 
all except a finite number of zeros lie in the union of angles $\{|\arg z| <\vep\} 
\cup \{|\arg z-\pi| <\vep\}$, then $ f(z) = \frac{a}{\sin^2(bz-c)} $ for some $a,b\ne 0$, $b\in\R$ and $c\in \co$.

There are numerous results concerning zero distribution for the 
solutions of differential equation \eqref{nz}. In particular, S.~Hellerstein, L.-C.~Shen and J.~Williamson
\cite{hel} showed that if there exist two linearly independent solutions of  \eqref{nz} which have only real zeros, 
then $Q$ is a constant. This result was generalized by F.~Br\"uggemann \cite{bru}
and N.~Steinmetz \cite{ste} who showed, in particular, that if the zeros of all solutions from a fundamental system 
of a linear equation $g^{(n)} + \sum_{j=0}^{n-1} Q_j g^{(j)} =0$ with polynomial coefficients are in some sense 
close to the real axis, then all the coefficients $Q_j$ are constants. However, it should be emphasized 
that it is essential that the condition holds for {\it all} solutions from a fundamental system. The simplest example of the 
solutions of $g'' -zg =0$ (Airy functions ${\rm Ai}$ and ${\rm Bi}$) shows that for nonconstant coefficients 
it may happen that one solution has only real zeros (see Section \ref{exa} for details). Moreover, 
this example shows that the bound on $\alpha$ in Theorem \ref{kr1} is sharp, since 
the function $1/{\rm Bi}^2$ admits a representation \eqref{mer2} and the zeros of ${\rm Bi}$ accumulate
to the rays $\{\arg z =\pi/3, -\pi/3, \pi\}$.

In Theorem \ref{kr1} we assume that only one solution of \eqref{nz} 
has zeros in some angle (whose size depends on the order), but we have an additional property  
that $\frac{1}{g^2}$ has an expansion
into a series of simple fractions. This allows us to use some results about the so-called Krein class of entire functions. 
\smallskip

The paper is organized as follows. In Section \ref{prel} we remind the basic definitions of the Nevanlinna theory
and prove some auxiliary estimates. Theorem \ref{main1} is proved in Section \ref{sect3}. In Section \ref{difeq} 
we consider the relations between zero-free functions of the form \eqref{mer2} and differential equations,
and prove Proposition \ref{diff} and Theorem \ref{kr1}. Some basic examples are discussed in Section \ref{exa}.
\medskip
\\
\textbf{Acknowledgement.} The authors are grateful to professors J.K.~Langley,  
J.~Heittokangas and K.G.~Malyutin for helpful comments and references
and to the referees for numerous helpful suggestions which improved the presentation.
\bigskip


\section{Preliminaries}
\label{prel}

In what follows we write $U(x)\lesssim V(x)$ if 
there is a constant $C$ such that $U(x)\leq CV(x)$ holds for all $x$ 
in the set in question. We write $U(x)\asymp V(x)$ if both $U(x)\lesssim V(x)$ and
$V(x)\lesssim U(x)$. The standard Landau notations
$O$ and $o$ also will be used. We use the function $\log^+$ defined as
$\log^+x = \max (\log x, 0)$, $x\ge 0$.

We recall some basic notations of the Nevanlinna theory (see, e.g., \cite{gold}). Given a meromorphic function $f$,
we denote by $n(r,f)$ the counting function of its poles and put
$$
N(r, f) = \int_0^r \frac{n(t, f) - n(0,f)}{t}dt + n(0,f) \log r.
$$
The characteristics  $N(r, f)$ and  $N(r, 1/f)$ account, respectively, for the number of poles and zeros of $f$. 
We also put 
$$
m(r,f) = \frac{1}{2\pi} \int_0^{2\pi} \log^+|f(re^{it})|\,dt
$$
and $T(r,f) =N(r, f) +m(r, f)$. Recall that the Nevanlinna characteristic $T(r,f)$ is an increasing function of $r$, 
$T(r, f) \asymp T(r, 1/f)$, and $\log r = o(T(r,f))$, $r\to \infty$, unless $f$ is a rational function.
The order of a meromorphic function is defined as
$$
\rho(f) = \limsup_{r\to \infty} \frac{\log T(r,f)}{\log r}.
$$
Finally, the defect of the value 0 is defined as
$$
\delta(0,f) = \liminf_{r\to \infty} \frac{m(r, 1/f)}{T(r,f)} = 1-  \limsup_{r\to \infty} \frac{N(r, 1/f)}{T(r,f)}.
$$

An important property of functions $f$ of the form \eqref{mer} satisfying \eqref{conv} 
used in the proof of the Keldysh theorem is the following estimate \cite[Chapter V, Theorem 6.1]{gold}:
for any $p\in(0,1)$,
\begin{equation}
\label{akir}
\int_0^{2\pi} |f(re^{i\phi})|^p d\phi= o(1), \qquad r\to \infty.
\end{equation}
We need an analog of this statement for the functions $f(z) = \sum_{n=1}^\infty \frac{c_n}{(z-t_n)^2}$
such that $\sum_n \frac{|c_n|}{|t_n|^2} <\infty$. One can expect that in this case \eqref{akir} holds for any $p\in (0,1/2)$, 
but we were not able to find a proof. However, one can show that \eqref{akir} holds even with $p=1$ for a sufficiently dense set 
of radii.

Let $\ttt=\{t_n\}$ have a finite convergence exponent and let $L>0$ be such that
$\sum_n |t_n|^{-L} <\infty$. Put 
\begin{equation}
\label{univ}
F = \bigcup_n \big[|t_n| - |t_n|^{-L}, |t_n| + |t_n|^{-L}\big].
\end{equation}
Finally, for $R>0$ put
$$
\mathcal{A}_R=\big\{ z\in \co:\ R<|z|<2R, \ |z| \notin F\big\}.
$$

\begin{proposition}
\label{inte}
Let $f(z) = \sum_{n=1}^\infty \frac{c_n}{(z-t_n)^2}$ where $\ttt=\{t_n\}$ has a finite convergence exponent 
and $\sum_{n\in \mathbb{N}}\frac{|c_n|}{|t_n|^2}<+\infty$. Then
there exists	$\{r_j\}_{j\in \mathbb{N}}$ satisfying $r_{j+1}\ge 4r_j$ and $\ln r_j \asymp j$  such that  
\begin{equation}
	\int \limits_{0}^{2\pi} |f(r_j e^{i\phi})|\, d\varphi =o(1), \qquad j\rightarrow \infty.
\label{noct}
\end{equation}
\end{proposition}

In the proof we will need the following elementary lemma.

\begin{lemma}
\label{elem}
		Let $a_k\ge 0$ and $\sum\limits_{k\ge 1} a_k <\infty$. Then there exists a subsequence $\{a_{k_n}\}$
		of $\{a_k\}$ with $a_{k_n}=o\Big(\frac{1}{k_n}\Big)$ and $n\le k_n \le 2n$.
\end{lemma}
	
The proof of the lemma is immediate, since for any sufficiently large $m$ the inequality 
$$
	a_k \le \frac{1}{k} \Big( \sum_{j=2^m+1}^{2^{m+1}} a_j\Big)^{1/2}
$$
holds for at least $2^{m-1}$ values of $k\in [2^m+1, 2^{m+1}]$. It remains to incorporate these $2^{m-1}$ values
of $a_k$ into $\{a_{k_n}\}$.
	
\begin{proof}[Proof of Proposition \ref{inte}]
Let $L$, $F$ and $\mathcal{A}_R$ be as above. Consider the integral
$$
\begin{aligned}
\mathcal{J}(R) & = \int\limits_{\mathcal{A}_R}\bigg|\sum\limits_{n\ge 1}\frac{c_n}{(z-t_n)^2} \bigg|\,dxdy 
 = 
\int\limits_{\mathcal{A}_R}\bigg|\sum\limits_{|t_n| \le R/2}\frac{c_n}{(z-t_n)^2} \bigg|\,dxdy\\
&  +
\int\limits_{\mathcal{A}_R}\bigg|\sum\limits_{R/2< |t_n| <4R}\frac{c_n}{(z-t_n)^2} \bigg|\,dxdy  +
\int\limits_{\mathcal{A}_R}\bigg|\sum\limits_{|t_n| \ge 4R}\frac{c_n}{(z-t_n)^2} \bigg|\,dxdy \\
&  =  \mathcal{J}_1(R) + \mathcal{J}_2(R) +\mathcal{J}_3(R).
\end{aligned}
$$
We show that for a sufficiently dense set of $R$-s one has $\mathcal{J}(R)  = o(R^2)$, $R\to \infty$.

Since $|z-t_n|\asymp R$ when $z\in \mathcal{A}_R$, $|t_n| \le R/4$, we have
$$
\mathcal{J}_1(R) \lesssim  \sum\limits_{|t_n| \le R/2} |c_n| \lesssim 
R \sum\limits_{|t_n|\le R^{1/2} } \frac{|c_n|}{|t_n|^2} + R^2 \sum\limits_{R^{1/2} \le |t_n| \le R/2} \frac{|c_n|}{|t_n|^2} = o(R^2).
$$
Here we used the condition $\sum_{n\in \mathbb{N}}\frac{|c_n|}{|t_n|^2}<+\infty$. Analogously, 
$$
\mathcal{J}_3(R) \lesssim  R^2 \sum\limits_{|t_n| \ge 4R} \frac{|c_n|}{|t_n|^2} = o(R^2).
$$
Note that if $R/2 < |t_n| <4R$ and $z\in \mathcal{A}_R$, then we have
$|t_n|^{-L} < |z-t_n| <6R$. The first inequality follows from the fact that $z\notin F$.
Since $\int_{\vep <|z| <r} |z|^{-2} dxdy \asymp \log \frac{r}{\vep}$ we conclude that
$$
\mathcal{J}_2(R) \lesssim \log R \sum\limits_{R/2 < |t_n| <4R}  |c_n| 
\lesssim R^2 \log R \sum\limits_{R/2 < |t_n| <4R}  \frac{|c_n|}{|t_n|^2}.
$$
		
Now put 	
$$
	a_k=\sum_{8^{k}<|t_n|<8^{k+1}}\frac{|c_n|}{|t_n|^2}.
$$
Then $\sum\limits_{k\ge 1} a_k<\infty$ and by Lemma \ref{elem} there exists a subsequence $a_{k_j}$ 
such that $a_{k_j}=o\left(\frac{1}{k_j}\right)$ and $j\le k_j \le 2j$. Let us show that for the sequence
$R_j=2^{3k_j+1}=2\cdot 8^{k_j}$ we have $\mathcal{J}(R_j) = o(R^2_j)$.  
Indeed, $k_j\asymp \ln R_j$ and
$$
\sum\limits_{R_j/2 <|t_n|<4R_j}\frac{|c_n|}{|t_n|^2}=	
\sum\limits_{8^{k_j}<|t_n|<8\cdot8^{k_j}}\frac{|c_n|}{|t_n|^2}=a_{k_j}=
o\left(\frac{1}{k_j}\right)=o\left(\frac{1}{\ln R_j}\right).
$$
Therefore, 
$$
\mathcal{J}_2(R_j) \lesssim R^2_j\ln R_j \sum\limits_{R_j/2 <|t_n|<4 R_j}\frac{|c_n|}{|t_n|^2}=o(R^2_j).
$$
Thus, $\mathcal{J}(R_j) = o(R^2_j)$ and $\ln R_j \asymp k_j \asymp j$.		

Since 
$$
\mathcal{J}(R_j)  = \int\limits_{[R_j, 2R_j]\setminus F} 
\bigg(\int \limits_{0}^{2\pi} |f(r e^{i\phi})|\, d\varphi \bigg) r\,dr  = o(R^2_j),
$$
we conclude that one can choose a sequence $r_j \in [R_j, 2R_j]$ such that \eqref{noct} holds. 
Note that $r_j \in [2^{3k_j+1}, 2^{3k_j+2}] \subset  [2^{3j+1}, 2^{6j+2}]$ since $j\le k_j\le 2j$.
Thus, $\ln r_j \asymp j$.
\end{proof}

To simplify the presentation, in the proof of  Proposition \ref{inte} we considered the dyadic annuli. Clearly, 
we can make the sequence $r_j$ to be arbitrarily dense in the logarithmic scale. 
Repeating the proof of  Proposition \ref{inte} with $2$ replaced by $1+\delta$ with $\delta>0$, we get

\begin{proposition}
\label{inte1}
Let $\ttt=\{t_n\}$ have a finite convergence exponent, $\sum \limits_{n\in \mathbb{N}}\frac{|c_n|}{|t_n|^2}<+\infty$
and let $\delta>0$. Then there exists	$\{r_j\}_{j\in \mathbb{N}}$ satisfying $r_{j+1}\ge (1+\delta) r_j$ and 
$$
(1+\delta)^{3j+1} \le r_j \le (1+\delta)^{3(1+\delta)j+2}
$$
such that  for the function $f(z) = \sum_{n=1}^\infty \frac{c_n}{(z-t_n)^2}$ we have
\begin{equation}
	\int \limits_{0}^{2\pi} |f(r_j e^{i\phi})|\, d\varphi =o(1), \qquad j\rightarrow \infty.
\label{noct1}
\end{equation}
\end{proposition}
	
Finally, let us show that the order of a function $f(z) = \sum_{n=1}^\infty \frac{c_n}{(z-t_n)^2}$ is completely 
determined by its poles.
		
\begin{lemma}
\label{ord}
Let $\ttt=\{t_n\}$ have a finite convergence exponent, $\sum \limits_{n\in \mathbb{N}}\frac{|c_n|}{|t_n|^2}<+\infty$ and $f(z) = \sum_{n=1}^\infty \frac{c_n}{(z-t_n)^2}$. Then
$$
\limsup_{r\to \infty} \frac{\log N(r,f)}{\log r} = \limsup_{r\to \infty} \frac{\log T(r,f)}{\log r}.
$$
\end{lemma}

\begin{proof}
Let $|z| =r \notin F$, where $F$ is defined by \eqref{univ}.
Then for $r/2 <|t_n| <2r$ one has $|z-t_n|\gtrsim r^{-L}$ and
$$
\begin{aligned}
|f(z)| & \lesssim \sum\limits_{|t_n| \le r/2}\frac{|c_n|}{|t_n|^2} + r^L \sum\limits_{r/2<|t_n| <2r} |c_n| +
\sum\limits_{|t_n| \ge 2r} \frac{|c_n|}{|t_n|^2}  \\ 
& \le 4r^{L+2} \sum\limits_{r/2<|t_n| <2r} \frac{|c_n|}{|t_n|^2} +O(1) =o(r^{L+2}).
\end{aligned}
$$
Thus, $m(r,f) = O(\log r) = o(T(r,f))$ and $N(r,f) \sim T(r,f)$ as $r\to \infty$, $r\notin F$. Note that $F$ has a finite measure and $T(r,f)$ 
is increasing. Hence, for any $r\in F$ there exists $\tilde r >r$ such that $\tilde r \notin F$ and $\tilde r-r \lesssim 1$. Then
$$
\limsup_{\stackrel{r\to \infty}{r\in F}} \frac{\log T(r,f)}{\log r} 
\le \limsup_{\stackrel{\tilde r\to \infty}{\tilde r\notin F}} \frac{\log T(\tilde r,f)}{\log \tilde r} = 
\limsup_{\stackrel{\tilde r\to \infty}{\tilde r\notin F}}  \frac{\log N(\tilde r,f)}{\log \tilde r}.
$$
\end{proof}
\bigskip


\section{Proof of Theorem \ref{main1}}
\label{sect3}
	
The proof of Theorem \ref{main1} essentially follows the idea of the proof of the Keldysh theorem in
\cite[Chapter V, Theorem 6.2]{gold}. The key step is to show that $\frac{1}{|f(z)|}\asymp |z|^2$ on a sequence 
of circles outside the sets of small angular measure. 

Without loss of generality, let $\sum \limits_{k\in \mathbb{N}}c_k=1$. Then we have
\[
z^2f(z)-1=
\sum\limits_{k \ge 1} c_k\bigg(\frac{z^2}{(z-t_k)^2} -1\bigg) = 
\sum\limits_{k \ge 1}\frac{2c_kt_k}{z-t_k}+\sum\limits_{k \ge 1}\frac{c_kt_k^2}{(z-t_k)^2}.
\]
Let $r_j$ be the sequence of the radii from Proposition \ref{inte}, that is, $r_{j+1}\ge 4r_j$, $\log r_j \asymp j$, and 
 \[
 \int \limits_{0}^{2\pi}\bigg|\sum\limits_{k\ge 1}\frac{c_kt_k^2}{(r_ je^{i\varphi}-t_k)^2} \bigg| d\varphi =o(1), 
 \qquad j\to \infty.
\]
Also, by \eqref{akir}, 
 \[
 \int \limits_{0}^{2\pi}\bigg|\sum\limits_{k\ge 1}\frac{2c_kt_k}{r_je^{i\varphi}-t_k} \bigg|^{1/2} d\varphi =o(1),
 \qquad j\to \infty.
 \]
 Therefore, there exist sets $E_j\subset  [0, 2\pi]$ such that $|E_j| \to 0$ and $|z^2f(z)-1| <1/2$ when
 $z = r_j e^{i\phi}$ and $\phi \in [0, 2\pi] \setminus E_j$. 
 Here and it what follows we denote by $|E|$ the Lebesgue measure  of a measurable set $E\subset \R$.
 Thus, 
 \begin{equation}
 \label{ej}
 \frac{1}{|f(r_j e^{i\phi})|}\le 2r_j^2, \qquad \phi \in [0, 2\pi] \setminus E_j.
 \end{equation}
 Now we can apply the following estimate due to A.~Edrei and W.H.J.~Fuchs \cite[Chapter I, Theorem 7.3]{gold}:
 for any meromorphic function $g$, any measurable set $E \subset  [0, 2\pi]$ with $|E|<1/2$ and $k>1$ one has
 \begin{equation}
 \label{edre}
 \int_{E} \ln^+|g(re^{i\phi})|\, d\phi \le C(k)|E|\log\frac{1}{|E|} T(kr, g).
 \end{equation}
 We apply this estimate to $g=1/f$, $r=r_j$ and $k=4$ to obtain
 $$
 \begin{aligned}
 m(r_j, 1/f) & = \int_{[0,2\pi] \setminus E_j}\log^+ \frac{1}{|f(r_j e^{i\phi})|} \, d\phi +
 \int_{E_j}\log^+ \frac{1}{|f(r_j e^{i\phi})|} \, d\phi \\ 
 & \lesssim \log r_j + |E_j|\log\frac{1}{|E_j|} T(4r_j, 1/f) =o(T(4r_j, 1/f)).
 \end{aligned}
 $$
 
Assume that $\delta(f,0) = \liminf_{r\rightarrow \infty} \frac{m(r, 1/f)}{T(r,f)} >0$. 
Then there exists $c>0$ such that, for sufficiently large $j$, we have  $m(r_j, 1/f)\ge c T(r_j,f)$.
Also, for any $\vep>0$, we have $m(r_j, 1/f) \le \varepsilon \cdot T(4r_j,f)$ for $j \ge j_0 = j_0(\vep)$.
Combining these inequalities and using the fact that $r_{j+1} \ge 4r_j$  we get 
\[
	T(r_{j+1}, f) \ge T(4r_j,f) \ge \frac{c}{\varepsilon}\, T(r_j,f), \qquad j\ge j_0.
\]
Hence,
\[
	T(r_j,f) \ge \Big(\frac{c}{\varepsilon} \Big)^{j-j_0} T(r_{j_0},f), \qquad j\ge j_0.
\]
Recall that $\log r_j \asymp j$ with some absolute constants. Therefore, 
$$
	\rho(f)=\limsup\limits_{r\rightarrow \infty} \frac{\log T(r,f)}{\log r} \ge 
	\limsup \limits_{j\rightarrow \infty} \frac{\log T(r_j,f)}{\log r_j} \ge
	\limsup \limits_{j\rightarrow \infty} \frac{j-j_0}{\log r_j}\, \log \frac{c}{\vep}\ge c_1 \log \frac{c}{\vep}
$$	
for some numeric constant $c_1>0$. Since $\vep>0$ is arbitrary, we conclude that $\rho(f)= \infty$, a contradiction. 

It remains to show that the orders of $T(r,f)$ and $N(r, 1/f)$ coincide. The orders of  $T(r,f)$ and $N(r, f)$ 
coincide by Lemma \ref{ord}. Assume that 
$$
\tilde \rho =  \limsup_{r\to \infty} \frac{\log N(r,1/f)}{\log r} < \rho(f)
$$
and choose $\mu_0, \mu_1, \mu_2$ so that $\tilde \rho < \mu_2<\mu_1<\mu_0 <\rho(f)$. Next we find $\delta>0$ 
such that $\mu_0>(1+\delta) \mu_1$, $\mu_1>(1+\delta) \mu_2$. Finally, for this $\delta$ we find a sequence
$r_j \in [(1+\delta)^{3j+1}, (1+\delta)^{3(1+\delta)j+2}]$ 
from Proposition \ref{inte1} and obtain \eqref{ej} for some sets $E_j$ with 
$|E_j| \to 0$. Applying the Edrei--Fuchs estimate to $1/f$ 
and $k=(1+\delta)^3$ we conclude that $m(r_j, 1/f) = o(T(kr_j, 1/f))$, $j\to \infty$ and so
$$
T(r_j, f)  = T(r_j, 1/f) +O(1) \le \vep_j T(kr_j, 1/f) + N(r_j, 1/f),
$$
where $\vep_j = \vep_j (\delta) \to 0$, $j \to \infty$. Without loss of generality we may assume in what follows that
$\vep_j <(4k^{\mu_1})^{-1}$ and $N(r_j, 1/f) <r_j^{\mu_2}$  for any $j$.

Let $R$ be a sufficiently large number such that $T(R,f) >2R^{\mu_0}$
and choose $j$  such that $R \in [r_{j-1}, r_{j}]$. Then
$$
R^{\mu_0} \ge r_{j-1}^{\mu_0} \ge (1+\delta)^{(3j-2)\mu_0}, \qquad
r_j^{\mu_1} \le (1+\delta)^{(3(1+\delta)j+2)\mu_1}.
$$
It follows from the inequality $\mu_0>(1+\delta) \mu_1$ that 
$$
T(r_j, f)  \ge T(R, f) > 2R^{\mu_0} > 2r_j^{\mu_1}
$$
if $R$ (hence, $j$) is sufficiently large. 
Since $r_{j+l}^{\mu_2} \le (1+\delta)^{(3(1+\delta)(j+l)+2)\mu_2}$ and 
$k^{l\mu_1}r_j^{\mu_1} \ge (1+\delta)^{(3l+3j+1)\mu_1}$, we can also choose 
$j$ so that
 \begin{equation}
 \label{doha}
k^{l\mu_1}r_j^{\mu_1} > r_{j+l}^{\mu_2}, \qquad l \ge 1.
 \end{equation}
 
The rest of the proof is analogous to \cite[Chapter V, Theorem 6.2]{gold}. We have
$T(r_{j+1}, f) \ge \frac{1}{\vep_j} (T(r_j, f) - N(r_j, 1/f))$, whence
$$
T(r_{j+1}, f) \ge 4k^{\mu_1}(2r_j^{\mu_1} - r_j^{\mu_2}) \ge 4k^{\mu_1} r_j^{\mu_1}.
$$
Since $r_j^{\mu_2} <T(r_j, f)/2$, we also have
$$
T(r_{j+1}, f) \ge 4k^{\mu_1}(T(r_j, f) - T(r_j, f)/2) = 2k^{\mu_1}T(r_j, f).
$$
By induction, making use of  \eqref{doha}, we obtain
$$
T(r_{j+l}, f) \ge \max\big(  4k^{l\mu_1}r_j^{\mu_1},  (2k^{\mu_1})^l T(r_j, f)\big).
$$
Since $r_{j+l} \le (1+\delta)^{3(1+\delta)(j+l) +2}$, we finally get 
$$
\limsup_{r\to\infty} \frac{\log T(r, f)}{\log r} \ge
\limsup_{l\to\infty} \frac{\log T(r_{j+l}, f)}{\log r_{j+l}} \ge \limsup_{l\to\infty}  \frac{l\log 2}{\log r_{j+l}} \ge
 \frac{\log 2}{3(1+\delta) \log(1+\delta)}.
$$
Choosing arbitrarily small $\delta$ we conclude that $f$ is of infinite order, a contradiction.
\qed
\bigskip


\section{Zero-free meromorphic functions and differential equations}
\label{difeq}

In this section we prove the statements relating functions of the form \eqref{mer2} with finitely many
zeros (or no zeros at all) and differential equations with polynomial coefficients.

\begin{proof}[Proof of Proposition \ref{diff}]
The meromorphic function $f/P$ has no zeros. Hence, $P/f$ is an entire function with zeros $\{t_n\}$ of multiplicity 
2. Therefore, there exists an entire function $g$ with simple zeros $\{t_n\}$ such that $P/f = g^2$.

Expanding $P/g^2$ near a point $t_n$ we get
$$
\begin{aligned}
f(z) & = \frac{P(z)}{g^2(z)} = \frac{P(t_n) + P'(t_n) (z-t_n) + O((z-t_n)^2)}{(g'(t_n)(z-t_n) + 
\frac{g''(t_n)}{2}(z-t_n)^2 +O((z-t_n)^3))^2} \\
& = \frac{P(t_n) + P'(t_n) (z-t_n) + O((z-t_n)^2)}{(g'(t_n))^2(z-t_n)^2} \Big(1-\frac{g''(t_n)}{g'(t_n)}(z-t_n) + O((z-t_n)^2)\Big) \\
& = \frac{P(t_n)}{(g'(t_n))^2}\cdot \frac{1}{(z-t_n)^2} +
\frac{P'(t_n) g'(t_n) - P(t_n)g''(t_n)}{(g'(t_n))^3} \cdot\frac{1}{z-t_n} + O(1).
\end{aligned}
$$
Since $f$ has zero residue at $t_n$, we conclude that $P'(t_n) g'(t_n) - P(t_n)g''(t_n)=0$ for all $n$ whence
$$
Pg'' - P'g' +Qg=0
$$
for some entire function $Q$. It remains to show that $Q$ is a polynomial.
Indeed, 
$$
Q  = P'\frac{g'}{g} - P\frac{g''}{g}.
$$
Since $f$, and hence $g$, is of finite order we have $m\big(r,\frac{g'}{g}\big) = O(\log r)$ 
and  $m\big(r,\frac{g''}{g}\big) = O(\log r)$ as $r\to \infty$. Therefore, $m(r,Q) = O(\log r)$ 
and so $Q$ is a polynomial. 
\end{proof}

\begin{remark}
{\rm Conversely, if an entire function $g$ with simple zeros $\{t_n\}$ satisfies the differential equation \eqref{fz} 
for some polynomials $P, Q$ and 
$$
\sum_n \frac{|P(t_n)|}{|g'(t_n)|^2(|t_n|^2+1)} <\infty,
$$
then 
$$
\frac{P(z)}{g^2(z)} = \sum_n \frac{P(t_n)}{(g'(t_n))^2(z-t_n)^2} +h(z),
$$
where $h$ is an entire function. To obtain a representation of $P/g^2$ in the form  \eqref{mer2} 
one has to prove that $h\equiv 0$. This is true, e.g., if 
$g$ is of finite order, its zeros lie on (or approach to) a finite system of rays and $g^2/P$ tends to infinity along any ray $L$
which is not parallel to a ray from this system.
In this case $h$ will tend to $0$ along $L$ and, thus, $h\equiv 0$ by the
Phragm\'en--Lindel\"of principle. For an example of such situation see the proof of the
expansion of the function $1/{\rm Bi}^2$ in Section \ref{exa}. }
\end{remark}

\begin{proof}[Proof of Corollary \ref{fer}]
Assume that $f$ has no zeros. Then $f=1/g^2$ for an entire function $g$ 
satisfying the equation 
$g'' +Qg=0$ for some polynomial $Q$ with degree $m$.
By a theorem of S.B.~Bank and I.~Laine \cite[Theorem 1]{bl} in the case when $m$ is odd
the convergence exponent of the zeros of $g$ (i.e., of the sequence $\mathcal{T} = \{t_n\}$)
equals $\frac{m+2}{2}$, while for an even $m$ 
either the convergence exponent of $\mathcal{T}$ is $\frac{m+2}{2}$ or $g$ has finitely many zeros. The latter case is impossible 
since the set $\mathcal{T}$ is assumed to be infinite. 
\end{proof}

In the proof of Theorem \ref{kr1} we will need the properties 
of entire functions whose inverses are represented as series of Cauchy kernels.
In 1947 M.G. Krein proved the following important result 
(see \cite[Theorem 4]{krein47} or \cite[Lecture 16]{lev}):
{\it Assume
that $g$ is an entire function, which is real on $\RR$,
with simple real zeros $t_n \ne 0$ and such that,
for some integer $k\ge 0$, we have
$$
\sum_n \frac{1}{|t_n|^{k+1} |g'(t_n)|}<\infty
$$
and 
\begin{equation}
\label{krein}
\frac{1}{g(z)} = \sum_n
\frac{1}{g'(t_n)} \cdot \bigg(\frac{1}{z-t_n} +\frac{1}{t_n}+
\cdots + \frac{z^{k-1}}{t_n^k} \bigg) + R(z), 
\end{equation}
where $R$ is some polynomial. 
Then $g$ is a function of finite exponential type and, moreover, of the Cartwright class.}
Krein \cite[Theorem 5]{krein47} (see also \cite[Lecture 16]{lev}) showed also that the condition 
$t_n\in\mathbb{R}$ can be relaxed to the Blaschke condition
$\sum_n |t_n|^{-2} |\ima t_n| <\infty$.

V.B.~Sherstyukov \cite{shers} extended this result as follows: {\it If $g$ is given by \eqref{krein}, where the set
$\{t_n\}$ is contained in some strip and has
a finite convergence exponent, then $g$ is a function of finite exponential type.} In \cite{abb} this result was 
generalized to entire functions representable as a ratio of two (regularized) Cauchy transforms. 

\begin{proof}[Proof of Theorem \ref{kr1}]
Assume that $f$ has no zeros. Then,
by Proposition \ref{diff}, $f=1/g^2$ where $g$ is an entire function of finite order with simple zeros $t_n$ 
and $g''+Qg=0$ for some polynomial $Q$. Since $\rho (g) =\rho(1/f) = \rho(f) = \frac{m+2}{2}$,  we have 
${\rm deg}\, Q =m$. If $m\ne 0$, then the zeros $t_n$ of $g$ 
approach the  critical rays (see \eqref{appr1} and \eqref{appr}).  Since the critical rays have minimal
angular distance $\frac{2\pi}{m+2}$
and all points from $\ttt$ except a finite number lie in the union of the angles
$\{|\arg z| <\alpha\} \cup \{|\arg z -\pi| <\alpha\} $ for some $\alpha\in (0, \frac{\pi}{m+2})$,
we conclude that $t_n$ approaches at most two of the critical rays. If $m$ is even, they are obviously two parts
of the same straight line. However, if $m$ is odd and the set $\{t_n\}$ accumulates to two critical rays, then it omits
exactly $m$ critical sectors, a contradiction since the number of omitted sectors must be even \cite[Theorem 5.4]{laine}.
Thus, in the case of odd $m$ the set $\{t_n\}$ acumulates to one critical ray only. Thus, it follows from
\eqref{appr} that  in any case the set
$\{t_n\}$ approaches some line and so $\{t_n\}$ is contained in some strip.

We also have 
$$
c_n = \frac{1}{(g'(t_n))^2} \qquad \text{and} \qquad \sum_n \frac{1}{|t_n|^2 |g'(t_n)|^2 } <\infty.
$$
Since $\{t_n\}$ has a finite convergence exponent, there exists $M\in\N$ such that $\sum_n \frac{1}{|t_n|^M |g'(t_n)|}<\infty$
and so
$$
\frac{1}{g(z)} = 
\sum_n \frac{1}{g'(t_n)} \cdot \bigg(\frac{1}{z-t_n} +\frac{1}{t_n}+
\cdots + \frac{z^{M-1}}{t_n^M} \bigg) +h(z),
$$
where $h$ is some entire function. 

We need to show that $h$ is a polynomial. Since $g$ is of finite order and  $\{t_n\}$ has a finite convergence 
exponent it is easy to show that $h$ is of finite order. 
Since $\{t_n\}$ lie in some strip $\Pi$ of width $L$, it is clear that $|f(z)| \lesssim |z|^2 +1$, $z\notin \widetilde{\Pi}$,
where $ \widetilde{\Pi}$ denotes the strip with the same central line and the width $2L$.
Indeed, let $|z|=r$, $z\notin  \widetilde{\Pi}$. Then
$$
|f(z)| \lesssim \sum_{|t_n|\le r/2} \frac{|c_n|}{|t_n|^2} + \sum_{|t_n|\ge 2r} \frac{|c_n|}{|t_n|^2} + 
\frac{r^2}{L^2}\sum_{r/2< |t_n| <2r} \frac{|c_n|}{|t_n|^2} \lesssim r^2+1.
$$
Similarly, there exists $K \ge 1$ such that
$$
\bigg|\sum_n \frac{1}{g'(t_n)} \cdot \bigg(\frac{1}{z-t_n} +\frac{1}{t_n}+
\cdots + \frac{z^{M-1}}{t_n^M} \bigg)\bigg| \lesssim |z|^K+1, \qquad z\notin  \widetilde{\Pi}.
$$
Thus, 
$$
(h(z) + O(|z|^K))^2 = O(|z|^2), \qquad z\notin  \widetilde{\Pi}, \ |z|\ge 1,
$$
and we conclude that $|h(z)| \lesssim |z|^K$ for $z\notin  \widetilde{\Pi}$. Since $h$ is of finite order, $h$ is a polynomial 
by the Phragm\'en--Lindel\"of principle. Thus, $g$ is of the form \eqref{krein} with $t_n$ in a strip and, by 
the result of Sherstyukov \cite{shers} cited above, $g$ is of order at most 1. Recall that $\rho(g) = \frac{{\rm deg}\, Q+2}{2}$,
and so $Q$ is a constant. Hence, $g = a\sin (bz-c)$ for some parameters $a,b,c$. 
\end{proof}
\bigskip


\section{Examples}
\label{exa}

Using \eqref{zfr} one can easily find examples where $f$ has a finite number of zeros by 
considering the function $\cos(z^N)$ with different $N$. E.g., 
if $u_n$ is one of the solutions of $u_n^2 = \frac{\pi}{2} +\pi n$, $n\in\Z$, then
$$
\frac{z}{\cos^2 (z^2)} = \sum_{n\in\Z} \frac{1}{4u_n}\bigg( \frac{1}{(z-u_n)^2} - \frac{1}{(z+u_n)^2} \bigg),
$$
while $g(z) = \cos(z^2)$ satisfies the equation
$$
zg'' - g' = -4z^3 g.
$$

Another example  is related to the equation $g''-zg = 0$, the simplest equation with a nonconstant 
coefficient. Two of its solutions are known as Airy functions ${\rm Ai}$  and ${\rm Bi}$ (see, e.g., \cite{as, olv}).
Both are entire functions of order $3/2$. All zeros
of ${\rm Ai}$ are real and negative while ${\rm Bi}$ has zeros $\{b_n\}_{n\ge 1}$ on the negative
semiaxis and two series of zeros $\{\beta_n\}_{n\ge 1}$ and  $\{\bar \beta_n\}_{n\ge 1}$ in $\{\pi/3 <\arg z <\pi/2\}$
and $\{-\pi/2 <\arg z <-\pi/3\}$ respectively, where $\bar \beta_n$ is the conjugate to $\beta_n$ and
the zeros in each series are ordered by increase of  the modulus.
The zeros $\beta_n$ approach the critical ray $\arg z= \pi/3$.

The function $\frac{1}{\rm Ai^2}$ cannot be represented as $\sum_n \frac{c_n}{(z-t_n)^2}$
since $|{\rm Ai}(z)|\to 0$ as $|z|\to \infty$ along any ray $\{\arg z = \alpha\}$ with $\pi/3<|\alpha| <\pi/2$.
Let us show that the function $\frac{1}{\rm Bi^2}$ does admit such expansion. It is known (see \cite[Chapter 10]{as}) that
$$
|b_n| \asymp |\beta_n| \asymp n^{2/3}, \qquad  |{\rm Bi}'(b_n)| \asymp   |{\rm Bi}'(\beta_n)|  \asymp n^{1/6}.
$$
Therefore the series $\sum_n |b_n {\rm Bi}'(b_n)|^{-2}$ converges and we can write
$$
\frac{1}{{\rm Bi}^2(z)} = \sum_n \bigg(
\frac{1}{({\rm Bi}'(b_n))^2(z-b_n)^2} + \frac{1}{({\rm Bi}'(\beta_n))^2(z-\beta_n)^2 }
+ \frac{1}{({\rm Bi}'(\bar \beta_n))^2(z- \bar \beta_n)^2 }
\bigg) +h(z)
$$
for some entire function $h$. It is easy to see that the function $h$ is of order at most $3/2$. Since 
$|{\rm Bi}(z)|\to \infty$ as $|z|\to \infty$ along any ray $\{\arg z = \alpha\}$, $\alpha \ne \pm \pi/3, \pi$, 
we conclude by the Phragm\'en--Lindel\"of principle that $h\equiv 0$.

This example shows that in the results of Hellerstein, Shen, Williamson or Steinmetz it is essential
that all elements of the fundamental system have their zeros near the real axis. It also shows the sharpness
of the conditions of Theorem \ref{kr1}.

\end{document}